\documentclass[twoside]{article}
\usepackage{jmlr2e}
\usepackage{amsfonts} 
\usepackage{subcaption}
\usepackage{graphicx}
\usepackage{algorithm}
\usepackage{algorithmicx}
\usepackage{algpseudocode}
\usepackage{amsmath,xcolor}
\usepackage{amsthm}
\usepackage{amsfonts}
\usepackage{enumitem}
\usepackage{amssymb}
\newcommand{\mb}{\mathbf}

\newcommand{\prox}{\text{prox}}
\newcommand{\proj}{\text{proj}}

\newcommand{\vA}{{\mathbf{A}}}
\newcommand{\vU}{{\mathbf{U}}}
\newcommand{\vV}{{\mathbf{V}}}
\newcommand{\vR}{{\mathbf{R}}}
\newcommand{\vb}{{\mathbf{b}}}

\newcommand{\vu}{{\mathbf{u}}}

\newcommand{\vx}{{\mathbf{x}}}
\newcommand{\vy}{{\mathbf{y}}}
\newcommand{\vz}{{\mathbf{z}}}

\newcommand{\PPM}{{\textsc{PPM}}}
\newcommand{\sign}{\mathbf{sign}}
\usepackage{amsmath}
\newcommand{\argmax}[1]{\underset{#1}{\mathrm{argmax}}}
\newcommand{\argmin}[1]{\underset{#1}{\mathrm{argmin}}}

\newtheorem{theorem}{Theorem}

\newtheorem{lemma}{Lemma}
\newtheorem{fact}{Fact}
\newtheorem{corollary}{Corollary}

\newtheorem{remark}{Remark}

\theoremstyle{definition}
\newtheorem{definition}{Definition}
\newtheorem{assumption}{Assumption}
\begin{document}

%

%

\title{Follow the flow: Proximal flow inspired multi-step  methods }

\author{\name Yushen Huang \email yushen.huang@stonybrook.edu \\
       \addr Department of Computer Science\\
       Stony Brook University
       \AND
       \name Yifan Sun \email yifan.sun@stonybrook.edu \\
       \addr  Department of Computer Science\\
       Stony Brook University}
\maketitle
\begin{abstract}
We investigate a family of approximate multi-step proximal point methods, accelerated by implicit linear discretizations of gradient flow. The resulting methods are multi-step proximal point methods, with  similar computational cost in each update as the proximal point method. We explore several optimization methods where applying an approximate  multistep  proximal points method results in improved convergence behavior. We argue that this is the result of the lowering of truncation error in approximating gradient flow. 
\end{abstract}
\section{Introduction}
In this paper, we consider the following optimization problem:
\begin{equation}\label{eq:obj}
    \min_{\vx \in \mathbb{R}^n} \; F(\vx) = f(\vx) + h(\vx)  
\end{equation}
where $f(\vx)$ is an $L$- smooth function, $h(\vx)$ is a closed convex but not neccessary smooth function and $F(\vx)$ is bounded below. 
The problem with the following settings has been raised in many applications ~\cite{tibshirani1996regression,yuan2006model,evgeniou2005learning,candes2012exact,friedman2008sparse}. 
In this paper, we consider a family of  multi-step proximal point updates.
The algorithm is a generalization of the proximal point method (PPM)~\cite{moreau1965proximite} where we use a linear combination of the previous $\tau$ steps in each iteration, as follows:
\begin{equation}
\tilde \vx^{(k)} = \sum_{i=1}^\tau \xi_i \vx^{(k-\tau+i)}, \qquad 
    \vx^{(k+1)}  = 
    \mathcal F( \tilde{\vx}^{(k)}) 
    \label{eq:main}
\end{equation}
Here, $\mathcal F$ is an approximate proximal point step. When $\tau = 1$ and $\xi_1 = 1$,  \eqref{eq:main} reduces to the ``vanilla" approximate proximal point method, of which there are many works \cite{moreau1965proximite,asi2019stochastic,nesterov2021inexact,asi2020minibatch}.
In this paper, we investigate improvements garnered by higher order $\tau > 1$. Note that unlike nonlinear explicit discretization methods (like Runge-Kutta), there is very little overhead in increasing $\tau$, as the averaging is done in an online manner. 

However, there are two questions that could arise naturally. 
First,  suppose we are given $\tau$; how do we choose $\xi_i$ optimally? Second, can  increasing $\tau$ always improve performance? 

For the first question, we link the multi-step proximal methods to the discretization of gradient flow. Using dynamical systems to interpret optimization methods has garnered considerable interest \cite{su2015differential,shi2019acceleration,zhang2018direct,romero2020finite}. 
To optimize the performance in practice, the coefficient should be chosen to minimize the truncation error, which leads to the so-called backward differential formula (BDF) Scheme. 
For the second question,
%
%
we show that in several important cases,
the methods give significantly better results, such as   proximal gradient in compressed sensing with both convex and nonconvex penalties,  alternating projections over  linear subspaces, and alternating minimization for matrix factorization. 
\subsection{Related Work}

\paragraph{Proximal point method} The proximal point method is originally from~\cite{moreau1965proximite} where one mimimizes the following  subproblem at each iteration; 
\[
\min_{\vy} F(\vy) + \frac{1}{2\beta} \Vert \vy - \vx \Vert_2^2.
\]
When $F$ is convex, this subproblem has a larger strong convexity parameter thereby facilitating faster numerical methods. In practice,   the subproblem will be solved approximately at each iteration; for example, using stochastic projected subgradient~\citep{davis2019proximally,asi2019stochastic}, prox-linear algorithm~\citep{drusvyatskiy2018error} and catalyst generic acceleration schema ~\citep{lin2015universal}. The proximal point method can also be generalized~\citep{nesterov2021inexact} by changing the penalty norm to $\Vert \cdot \Vert_2^{p+1}$ with $p \geq 1$ which have have faster convergence rates than the vanilla proximal point methods.

\paragraph{Dynamical systems inspired methods}  The idea behind the approach is that minimizing a function is equivalent to find the stationary point of a dynamical system. For example, \cite{su2014differential} analyzed Nesterov accelerated gradient descent method   as the discretization of a second order ordinary differential equation, and \cite{shi2019acceleration} analyzed the method formed by using a higher order symplectic discretization of a related differential equations  \cite{shi2021understanding}, to provide acceleration.   Others have attempted the same using  Runge-Kutta explicit discretizations \cite{zhang2018direct}. There are also work finding an optimization methods by discretizing other flow such as rescaled gradient flow \cite{wilson2019accelerating}.
A special example of discretization-inspired optimization improvements is the extragradient method \cite{korpelevich1976extragradient}, which is widely used in min-max optimization problems \citep{du1995minimax} and variational inequality problems \citep{facchinei2003finite}. 
The connection between dynamical systems ad optimization methods are also studied in other works~\citep{schropp2000dynamical,wibisono2015accelerated,krichene2015accelerated,orecchia2018accelerated,sundaramoorthi2018variational}.



\subsection{Contributions}
In this paper, we do the following.
\begin{itemize}[leftmargin=4mm]
    \item We propose a dynamical system that, when discretized either exlicitly or implicitly, leads to gradient descent, proximal gradient method, or the proimal point method. 

    \item We propose a higher order implicit discretization scheme, based on the backward differentiation formulas, which are computationally trivial extensions of existing methods, but whose performance can be significantly improved.
    
    \item We apply our methods to several problems, such as proximal gradient over  nonconvex sparse regularization, alternating minimization and alternating projections. Numerically, we see that in these scenarios, multistep methods perform an order of magnitude better than their vanilla versions. 

    \item We give a preliminary convergence analysis on smooth quadratic problems, as well as convex, strongly convex, and nonconvex smooth problems.

\end{itemize}
\section{Preliminary}
\subsection{Function conditions}

\begin{definition}
A function $f:\mathbb{R}^n\to \mathbb{R}$ is \emph{smooth} if for all $x$, $\nabla f(x)$ exists. 
It is additionally \emph{$L$-smooth} if its gradient is $L$-Lipschitz:
\[
\|\nabla f(x) - \nabla f(y)\|_2\leq L\|x-y\|_2, \quad \forall x,y.
\]
\end{definition}
An example of an application whose objective is smooth but not $L$-smooth is matrix factorization.
We now give a (relaxed) definition of convexity.
\begin{definition}
 $f$ is 
 \emph{$\mu$-convex at $\vx$}, if
    \[
f(\vy) \geq f(\vx) + \langle \nabla f(\vx),\vy-\vx\rangle + \frac{\mu}{2}\|\vy-\vx\|^2. \tag{$\mu$-convex} 
\]
for any $\vy$. 
If $f$ is $\mu$-convex for all points $\vx$, we say that $f$ is $\mu$-convex.
\end{definition}
Specifically, $\mu \geq 0$ implies $f$ is convex, and if $\mu < 0$ then the function may be nonconvex. 
Note that if $f$ is $L$-smooth, it is also $-L$-convex. 
However, the condition of  $\mu$-convex with a negative $\mu$ is more general; for example, the function $f(\vx) = \vx^4$  is not $L$-smooth, but is convex ($\mu = 0$).

\subsection{Dynamical System Approach for Optimization}
%
Consider \eqref{eq:obj} where $f(\vx)$ is a $L$-smooth and $h(\vx)$ is a convex function. 
The local minimum of $F(\vx)$ is a stationary point of the following differential equation (DE), which we term the \emph{proximal flow}
%
%
%
\begin{equation}\label{eq:prox_flow}
    \dot{\vx}(t) = \lim_{\beta \to 0} \frac{\prox_{\beta F}(\vx) - \vx}{\beta}.
\end{equation}
%
%
%
%
This DE has a close connection with the Moreau envelope of $F$
\[
F_{\beta}(\vx) = \inf_u\; (F(u) + \frac{1}{2\beta}\|x-u\|_2^2)
\]
whose gradient is 
\[
\nabla g_{\beta}(\vx) = \frac{1}{\beta} \left ( \vx - \prox_{\beta F} (\vx) \right).
\]
Note that this gradient exists and is unique for all  $0<\beta < \frac{L}{2}$.
Also, when $F$ is smooth, \eqref{eq:prox_flow} reduces to vanilla gradient flow
\[
\dot \vx(t) = -\nabla F(\vx(t))
\]

\subsection{Existance and uniqueness} 
To show the existance and uniqueness of \eqref{eq:prox_flow}, we will first show that its right-hand-side always exists, and that it is a special instance of subgradient flow.

\begin{lemma} 
Consider $F = f + h$ where $f$ is $L$-smooth and $h$ is convex. Then  the right-hand-side of \eqref{eq:prox_flow} always exists, and satisfies
\begin{equation}
\dot \vx(t) \in - \partial F(\vx(t))
\label{eq:subgradflow}
\end{equation}
where $\partial F(\vx(t))$ is the Clarke subdifferential of $F$ at $\vx(t)$ \citep{clarke1990optimization}
\[
\partial F(\vx):=\mathbf{conv}\left(\{g : \exists \vu_i  \to \vx, \nabla F(\vu_i)\to g\}\right)
\]
and is always a closed and convex set. 
\end{lemma}

\begin{proof}
By construction, 
 $F$ is $(\mu = -L)$-convex, and thus for all $0<\beta<1/L$, $F_\beta(\vx)$ is strongly convex in $\vx$, and a unique minimizer $\vx_\beta$ always exists; moreover, 
\begin{equation}
\frac{\vx_\beta-\vx}{\beta}\in - \partial F(\vx_\beta), \quad \forall 0 < \beta < \frac{1}{L}.
\label{eq:proxsol}
\end{equation}
Next, because $F$ is composed of a smooth and convex function, all its nondifferentiable points are isolated. So, assume that for some $\vx$, $\nabla F(\vx)$ exists; then \eqref{eq:proxsol} extends to $\beta = 0$. But if $\nabla F(\vx)$ does not exist, there still exists $\bar \beta$ small enough that, for all $0<\beta<\bar \beta$, $\nabla F(\vx_\beta)$ exists. By definition of the Clarke subdifferential, the limiting gradient $\lim_{\beta\to 0} \nabla F(\vx_\beta) \in \partial F(\vx)$.
 \end{proof}
System \eqref{eq:prox_flow} is a differential inclusion. Solutions of such systems exist when $\partial F$ is an upper hemicontinuous map, and is unique when it satisfies a one-sided Lipschitz condition.

\begin{lemma} 
For $F = f+g$ where $f$ is $L$-smooth and $g$ is convex, then the negative of the Clarke subdifferential \citep{clarke1990optimization}  
\[
\partial F(\vx):=\mathbf{conv}\left(\{g : \exists \vu_i  \to \vx, \nabla F(\vu_i)\to g\}\right)
\]
is always a closed and convex set, and is upper hemicontinuous
\[
\forall \vx, \; \exists \epsilon > 0,\;  \forall \vu:\|\vu-\vx\|_2\leq \epsilon, \quad  \partial F(\vu)\subset \partial F(\vx),
\]
and satisfies the one-sided Lipschitz condition, e.g.
\begin{equation}
 \langle g_x - g_y, \vx - \vy \rangle \leq L \Vert \vy - \vx \Vert^2,
\label{eq:onesidedL}
\end{equation}
for all $g_x\in -\partial F(\vx)$, $g_y\in -\partial F(\vy)$.
\end{lemma}

\begin{proof}
By construction, $\partial F(\vx)$ is always closed and convex. 
Note that $-\partial F(\vx)$ is upper hemicontinuous if $\partial F(\vx)$ is upper hemicontinuous. To see that this is true, note that $F = f + g$ where $\nabla f(\vx)$ always exists and is continuous everywhere, and  $\nabla g(\vx)$ exists for all but isolated points; moreover, by convexity, $g$ is locally Lipschitz. 
Therefore, for all but isolated points, $\nabla g(\vx)$ is also continuous. 
Assume that $\vx$ is differentiable; then there must exist some $\epsilon$ neighborhood around $\vx$ such that for all $\vu$, $\|\vu-\vx\|_2\leq \epsilon$, $\nabla F(\vu)$ exists, and $\nabla F(\vu)\overset{\vu\to \vx}{\to} \nabla F(\vx)$. 
Now assume that $\vx$ is not differentiable. Again, there is a neighborhood $\epsilon$ where for all $\|\vu-\vx\|_2\leq \epsilon$, $\nabla F(\vu)$ exists, and by definition, $\nabla F(\vu)\overset{\vu\to\vx}{\to}\nabla F(\vx)$.

On the other hand, because for any $g_x \in -\partial F(\vx), g_y \in -\partial F(\vy)$, we have
 \begin{align*}
 F(\vy) \geq F(\vx) - \langle g_x, \vy - \vx \rangle - \frac{L}{2} \Vert \vy - \vx \Vert^2 \\
  F(\vx) \geq F(\vy) - \langle g_y, \vx - \vy \rangle - \frac{L}{2} \Vert \vy - \vx \Vert^2  
 \end{align*}
 By combining the above equations, we recover \eqref{eq:onesidedL}.
As a result, the Clarke subdifferential is one sided Lipschitz.
\end{proof}
\begin{theorem}
    The proximal flow defined by \eqref{eq:prox_flow} always admits an existing and unique solution $\theta(t)$.
\end{theorem}
\begin{proof}
  Since \eqref{eq:prox_flow} is an instance of subgradient flow \eqref{eq:subgradflow}, and since the Clarke subdifferential is  upper hemicontinuous, then \eqref{eq:prox_flow} always admits an existing solution.
  Additionally, $-\partial F(\vx)$ satisfies the one-sided Lipschitz condition, and therefore the solution is unique.
\end{proof}

\subsection{Optimality}
Next, we show that the stationary point of the proximal flow (e.g. $\lim_{t \to \infty }\vx(t)$) is equivalent to find the stationary point of $F(\vx)$.

\begin{theorem}\label{the:equiv}
Consider the proximal flow in \eqref{eq:prox_flow} with $\vx(0) = \vx_0$.
Then $\lim_{t \to \infty} \vx(t)$ converges to stationary point of $F(\vx)$.
\end{theorem}
\begin{proof}
Recall that if $\dot\vx(t) = 0$, this implies
\[
 \lim_{\beta \to 0} \frac{\prox_{\beta F}(\vx) - \vx}{\beta} = 0 
\]
which is equivalent to 
\[
 0 = \lim_{\beta \to 0} -\nabla F_{\beta}(\vx) \in \lim_{\beta \to 0} \partial F(\prox_{\beta F}(\vx))  = \partial F(\vx).
\]
Hence $0$ is a subgradient of $\partial F(\vx)$ and thus $\vx$ is a stationary point of $F$.
\end{proof}
%
To give a more intuitive understanding of the proximal flow, we give the several examples of $F(\vx)$ and demostrate how to calculate its corresponding proximal flow.

\paragraph{Example: Smooth minimization.} Suppose $F(x) = f(x)$ is differentiable everywhere. Then, 
\[
\lim_{\beta \to 0} \frac{\prox_{\beta f}(\vx) - \vx}{\beta} =  -\nabla f(\vx)
\]
and the proximal flow reduces to gradient flow
\begin{equation}\label{eq:gradient_flow}
\dot{\vx}(t) = -\nabla f(\vx).
\end{equation}

\paragraph{Example: Smooth + nonsmooth.} 
Now suppose that $f$ is not only smooth, but is second-order continuous, and
where $h$ is nonsmooth but convex. Then, in the limit of $\beta \to 0$, we may use the Taylor approximation for $\vu \approx \vx$
\[
f(\vu) \approx   \nabla f(\vx)^T(\vu-\vx) 
\]
and therefore, in this limit,
\begin{multline*}
 \prox_{\beta f}(\vx) 
\\ =\argmin{\vu} \; \beta \nabla f(\vx)^T(\vu-\vx) + \beta h(\vu) + \frac{1}{2}\|\vx-\vu\|_2^2\\
 = \prox_{\beta h}(\vx-\beta \nabla f(\vx)).
\end{multline*}
In other words, in the limit of $\beta \to 0$, the proximal gradient method also reduces to the proximal point flow.

\paragraph{Example: Shrinkage.} Let us  consider $h(\vx) = \|\vx\|_1$. Then 
\[
 \prox_h (x)_i =  \begin{cases}
    x_i-\sign(x_i)\beta, & |x_i| > \beta,\\
    0 & \text{ else.}
\end{cases} 
\]
and therefore, in the limit $\beta\to 0$,
\[
\dot x_i =
  \begin{cases}
    -\sign(x_i), & |x_i| \neq 0,\\
    0 & \text{ else.}
\end{cases}
\]
Note that in this limit, $\dot x_i$ is not continuous (although $x_i$ is still continuous).
And, for all convex nonsmooth $h$,  by using the proximal operator and not the subgradient, we ensure that $\dot\vx(t)$ is always attaining a unique value.


 \paragraph{Example: LSP.} Let us  consider a ($\mu = -1$)-convex loss (e.g., nonconvex) $h(\vx) = \sum_i \log(1+b^{-1}|x_i|)$. Again, this proximal flow  may be computed element-wise. For $x_i\neq 0$, 
 \[
 \dot x_i = -\frac{\partial h(\vx)}{\partial x_i} = -\frac{\sign(x_i)}{b+|x_i|}.
 \]
 Now consider $x_i = 0$. Then 
 \[
\prox_{\beta F} (\vx)_i = \argmin{u} ~ \beta \log(1+|u|/b) + u^2 
 \]
 which is monotonically increasing in $u$
 for any value of $b > 0$, $\beta > 0$, and is minimized at  $u = 0$. Therefore,
\[
\dot x_i =  \begin{cases}
   -\frac{\sign(x_i)}{b+|x_i|}, & x_i > 0,\\
    0 & \text{ else.}
\end{cases}
\]

\subsection{Discretize Proximal Flow for Optimization}
%
%
%

\paragraph{Introducing prox grad and prox point}
At this point, we have now observed three different possible discretizations of the right hand side of \eqref{eq:prox_flow}; that is, we may update as 
\[
\frac{\vx_{k+1}-\tilde \vx_k}{\alpha} = \mathcal F(\tilde \vx_k)
\]
where for $\tilde \vx_k = \vx_k$, the choice of $\mathcal F(\vx_k)$ leads to three well-known methods
\begin{align*}
-\nabla F(\vx_k), && \text{(Gradient descent)}\\
\frac{\prox_{\alpha h}(\vx - \alpha \nabla f(\vx)}{\alpha}&& \text{(Proximal gradient method)} \\
\frac{\prox_{\alpha F}(\vx_k) - \vx_k}{\alpha} && \text{(Proximal point method).}
\end{align*}
where the last case reduces to the well-known proximal point method
\begin{equation} \label{eq:discretize_flow}
  \vx_{k+1} =  \text{prox}_{\alpha F }(\tilde{\vx}_k).
\end{equation}

\paragraph{Introducing the multi-steps}
We now consider the discretization of \eqref{eq:prox_flow}, in the case that $F$ is smooth everywhere. Then, we may write its \emph{implicit discretization} as
\[
\frac{\vx_{k+1}-\tilde \vx_k}{\alpha} = -\mathcal F(\tilde \vx_{k+1})
\]
where $\mathcal F(\vx) = \frac{\prox_{\beta F}(\vx) - \vx}{\beta}$ for a constant step size $\beta$.
Higher order implicit discretizations can then be imposed simply by modifying $\tilde \vx_k$:
\[
\tilde \vx_k = \sum_{i=1}^\tau \xi_i \vx_{k-\tau+i}
\]
While there exists a variety of methods for producing such discretizations, we focus on the \emph{backwards differential formula}, which explicitly dictates the constants $\xi_i$ as presented in Table \ref{tab:bdf-const}.
%
%
%
%
%
Specifically, these choices of $\xi$ are chosen to produce  the highest order of \emph{truncation} error where the \emph{truncation error} of gradient flow is defined below:
\begin{definition}
Let $\vx(t)$ be the solution of 
 a differential equation 
 \[
 \vx(t) = -\mathcal F(\vx(t))
 \]
 with $\vx(0) = \vx_0$. Now conisder an iterative update 
\[
\vx_{k+1} = \mathcal A(\{\vx_i\}_{i\leq k},\alpha).
\]
The (local) truncation error is defined as 
\[
\epsilon(\alpha) = \frac{\vx((k+1)\alpha) -  \mathcal A(\{\vx(i\alpha)\}_{i\leq k},\alpha)}{\alpha}.
\]
If $\lim_{\alpha \to 0}\epsilon(\alpha) = 0$, we call the iterative update is consistence.
In addition, if $\epsilon(\alpha) = \Theta(\alpha^{\tau})$, we say the iterative update $\mathcal A$ has truncation error of order $\tau$. 
\end{definition}
\begin{table}[]
\small
    \centering
\begin{tabular}{|c|c|cccc|}
\hline&&&&&\\ [-2ex]
     &  $\bar \xi$ & $\xi_1$ & $\xi_2$ & $\xi_3$ & $\xi_4$  \\ \hline
    BDF1 &1 &  1&&& \\
    BDF2 & 2/3 & -1/3 &  4/3 &&\\
    BDF3 & 6/11 & 2/11 & -9/11 &18/11 &\\
    BDF4 & 12/25& -3/25 & 16/25 & -36/25 & 48/25\\\hline
\end{tabular}
    \caption{Summary of constants for BDF methods.}
    \label{tab:bdf-const}
\end{table}
\section{Accelerated methods}

%
In this section, we empirically validate our proposed methods  by considering several optimization problems: proximal gradient with $\ell_1$ norm, proximal gradient with LSP penalty,  alternating minimization for matrix factorization, and alternating projection on linear subspaces.
For those experiments, we calculate the equation \eqref{eq:main} based on the following three approaches:
\paragraph{Multistep prox gradient} The idea of the first approach is based on approximating 
\begin{align}
\label{eq:proximal_mapping}
 \prox_{\alpha f + \alpha h}(\tilde \vx^{k}) = \argmin{\vx} \overbrace{\underbrace{\alpha f(\vx)  + \frac{1}{2}\|\vx-\tilde \vx^{(k,1)}\|_2^2}_{\text{smooth term}} + \underbrace{\alpha h(\vx)}_{\text{prox term}}}^{F_k(\vx)}.
\end{align}
This is done by first initializing $\tilde{\vx}^{k,1} = \tilde{\vx}^k$ based on left side of \eqref{eq:main} and perform $m$ iterations of the following update
\begin{align*}
\tilde \vx^{(k,j+1)} = 
\prox_{\beta \alpha h} (\tilde \vx^{(k,j)} - \alpha \beta \nabla f(\tilde \vx^{(k,j)})-\beta(\tilde \vx^{(k,j)}-\tilde \vx^{(k,1)} )).
\end{align*}
Then update $\vx^{(k+1)} = \tilde \vx^{(k,m+1)}$.

\paragraph{Multistep alternating minimization} The second approach is based on alternating minimization  \citep{shi2016primer,nesterov2012efficiency,liu2014asynchronous,fercoq2015accelerated}. 
The idea is to first interpret the (regularized) alternating minimization steps 
\begin{eqnarray*}
\vx_{1}^{(k+1)}& =& \argmin{\vx_1} \;f(\vx_1,\vx_2^{(k)}) + \frac{1}{2\alpha} \|\vx_1-\vx_1^{(k)}\|^2 \\
\vx_{2}^{(k+1)}& =& \argmin{\vx_2}\; f(\vx_1^{(k)},\vx_2 ) + \frac{1}{2\alpha} \|\vx_2-\vx_2^{(k)}\|^2\\
\end{eqnarray*}
as approximations of the  proximal point operations
\begin{multline*}
(\vx_{1}^{(k+1)} ,
\vx_{2}^{(k+1)}) \approx  \argmin{\vx_1,\vx_2} \; f(\vx_1 ,\vx_2 ) +\\ \frac{1}{2\alpha} \|\vx_1-\vx_1^{(k)}\|^2 + \frac{1}{2\alpha} \|\vx_2-\vx_2^{(k)}\|^2.
\end{multline*}
That is, the approximation is done by splitting the variable into  blocks and minimizing each block separately.

The multistep extension is then done by 
first initializing $\hat{\vx}^{(k,1)}_1,\hat{\vx}^{(k,1)}_2 = \tilde{\vx}^{(k)}$ 
and then performing 
$m$ iterations of 
\begin{align*}
                        \hat{\mathbf{x}}^{(k,j+1)}_1 &= \underset{\mathbf{x}_1}{\arg\min}f(\vx_1, \hat{\mathbf{x}}^{(k,j)}_2) 
                        + \frac{1}{2\alpha} \Vert \mathbf{x}_1 - \hat{\mathbf{x}}^{(k,j)}_1 \Vert^2 \\
                         \hat{\mathbf{x}}^{(k,j+1)}_2 &= \underset{\mathbf{x}_2}{\arg\min}f(\hat{\mathbf{x}}^{(k,j)}_1, \vx_2 )
                        + \frac{1}{2\alpha} \Vert \mathbf{x}_2 - \hat{\mathbf{x}}^{(k,j)}_2 \Vert^2.
\end{align*}
Then update $\vx^{(k+1)} = \hat \vx^{(k,m+1)}$. Although higher $m$ can often produce more precise proximal steps, in practice we find often $m = 1$ to be sufficient for acceleration.

\paragraph{Multistep alternating projections}
The third approach is to minimize the optimization
\[
 \min_{\vx} ~ \mathcal{I}_{ \mathcal{C}_1  \bigcap \mathcal{C}_2}(\vx)
\]
where $\mathcal C_1$ and $\mathcal C_2$ are closed linear subspaces. We do this by encoding 
\[
f(\vx) = 0, \quad h(\vx)
=  \left \{ \begin{array}{cc}
  0   & \vx \in \mathcal{C}_1  \bigcap \mathcal{C}_2  \\
  \infty  & \text{otherwise}
\end{array} \right.
\]
The approach is inspired by the alternating projection method \citep{von1949rings}  which performs the following updates, where   $\tilde{\vx}^{(k)}$ is constructed as the left hand side of \ref{eq:main}:
\begin{align*}
\vy^{(k+1)} &= \proj_{\mathcal{C}_1}(\tilde{\vx}^{(k)}) \\
\vx^{(k+1)} &= \proj_{\mathcal{C}_2}(\vy^{(k + 1)}).   
\end{align*}
Note that here, as in all cases, the major multistep extension is in the construction of $\tilde \vx^{(k)}$, as shown in \eqref{eq:main}. That is, the \emph{per-iteration complexity} of computing $\tilde \vx^{(k)}$ differently does \emph{not} significantly add computational overhead -- but it does significantly affect convergence results. 

\section{Numerics on accelerated methods}
We now discuss the methods producing the results in Figure \ref{fig:numerical}.
\subsection{Proximal gradient with 1-norm}
The proximal gradient with $\ell_1$ norm over compressed sensing problem is formulated as:
\begin{align*}
    \min_{\vx 
    \in \mathbb{R}^q} \frac{1}{2} \Vert \vA \vx - \vb \Vert^2 +\lambda \Vert \vx \Vert_1
\end{align*}
where $\vA \in \mathbb{R}^{p \times q}$ and $\vb \in \mathbb{R}^p$, with   $ p  \ll q$ (underdetermined system) and the $\ell_1$ norm is to make the solution of the system to be as sparse as possible. We choose $p = 100$ and $ q = 500$ and $\lambda = 0.1$. 
The entries of $\vA_{i,j},\vb_i\sim \mathcal N(0,1)$.
We use multistep proximal gradient to approximate the proximal mapping, with $m = 1$ and $m = 5$ inner iterations. 
We choose the maximum outer iteration to be 1000.
%
%
In this experiment (and many others) we see consistent improvement of   the higher-order BDF scheme. While the benefit seems small in this example, it becomes  more apparent when nonconvex regularization is used.
\subsection{Proximal gradient with LSP penalty}
The proximal gradient with LSP penalty over compressed sensing is formulated as follows:
\begin{align*}
    \min_{\vx 
    \in \mathbb{R}^q} \frac{1}{2} \Vert \vA \vx - \vb \Vert^2 +  \sum_{i=1}^q \log\left(1+\frac{|x_i|}{\lambda_i}\right)
\end{align*}
where  $\vA \in \mathbb{R}^{p \times q}$ and $\vb \in \mathbb{R}^p$. Similarly, $p \ll q$ since the system is underdetermined. 
Compared to the $\ell_1$ norm, the nonconvex LSP norm is a more aggressive sparsifier.
Usually, $\lambda_i = \lambda$ can be used to control the amount of concavity near 0, making sparsification more aggressive. In special cases, the weights $\lambda_i$ can also be varied to promote more sparsity in specific coordinates.
%

For the experiment settings, we choose $p = 100$ and $ q = 500$.
We choose $\lambda_i = 1$ for all $i$. 
For matrix $\vA$, we choose each entry of $\vA_{i,j}$ randomly by gaussian distribution with mean 0 and standard deviation 1.
For matrix $\vb$, we choose each entry $\vb_i$ by Gaussian distribution with mean 0 and standard deviation 1. 
We use  \textbf{Approach} 1 to approximate the proximal mapping with the number of inner approximations to be $m = 1$ and $m = 5$ and choose the maximum outer iteration to be 10000.
%
 Similar to $\ell_1$ norm setting, the higher-order BDF performs better than lower order BDF scheme.
\subsection{Alternating minimization for matrix factorization}
The matrix factorization problem can be formulated as the following:
\begin{align*}
    \min_{\vU,\vV}\; \frac{1}{2}\Vert \vU \vV^{T} - \vR \Vert_F^2
\end{align*}
The objective function is a classical ill-conditioned non-convex function.
For the experiment settings, we choose $\vU$ as a $100 \times 50$ matrix and $\vV$ as a $50 \times 100$ matrix. 
We choose $\vR  = \vU_{\text{true}} \vV_{\text{true}}^{T}$ where $\vU_{\text{true}}$ and $\vV_{\text{true}}$ have the same dimension as $\vU$ and $\vV$. 
We use our multistep alternating minimization method to approximate the proximal mapping with the number of inner approximations $m = 1$ and choose the maximum outer iteration to be 1000. (In the appendix, we also show a simulation for $m = 5$, whose results are incredibly similar to $m = 1$; the extra steps do not seem to offer much benefit.)
\subsection{Alternating projection over subspace}
Here, we take 
$\mathcal{C}_1$ and $\mathcal{C}_2$ to be the column space of matrix $\mathbf{C}_1 $, $\mathbf{C}_2\in \mathbb R^{30\times 20}$. Here we produce $\mathbf{C}_1$ where each element is $\sim\mathcal N(0,1)$. Then we generate $\mathbf{C}_2 = [\mb A, \mb B+\mb D]$ where $\mb A$ contains the first $20$ columns of $\mathbf C_1$, $\mb B$ contains the last 10 columns of $\mb C_1$, and  $\mb Z$ has elements $\sim\mathcal N(0,0.1)$, independently. 
\begin{figure*}[!ht]
    \centering
         \begin{subfigure}[b]{0.35\textwidth}
 \includegraphics[width=\linewidth, trim={2ex .5ex 2ex 1ex}, clip]{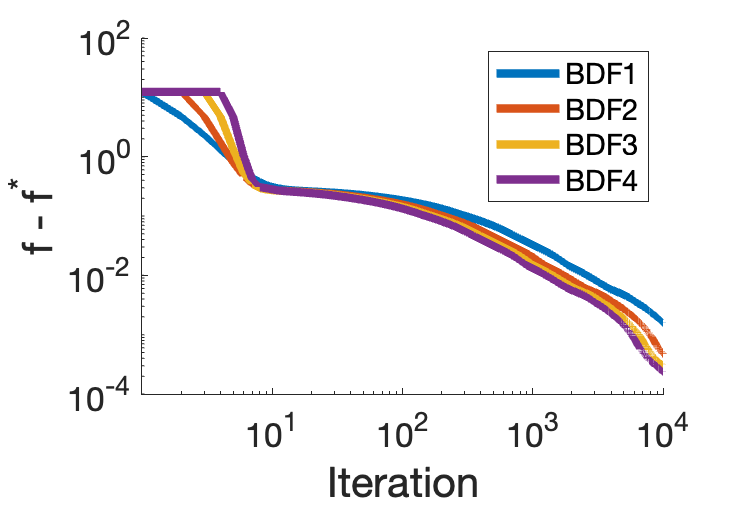}
\end{subfigure}
         \begin{subfigure}[b]{0.35\textwidth}
    \includegraphics[width=\linewidth, trim={2ex .5ex 2ex 1ex}, clip]{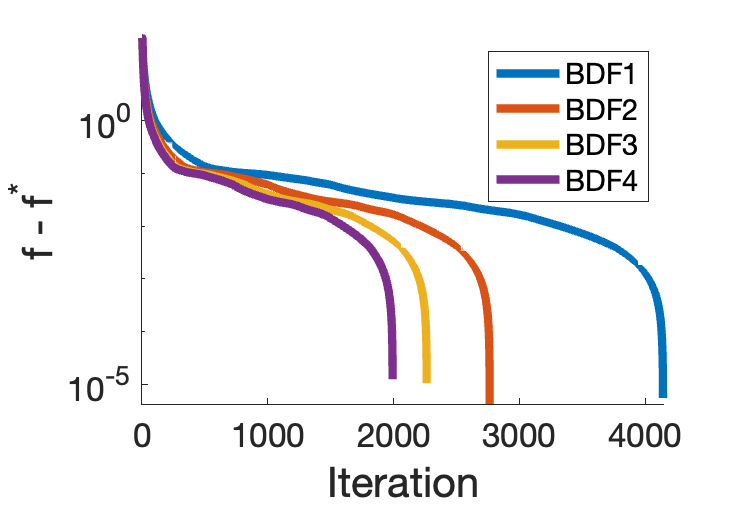}
\end{subfigure}\\
         \begin{subfigure}[b]{0.35\textwidth}    
    \includegraphics[width=\linewidth, trim={2ex .5ex 2ex 1ex}, clip]{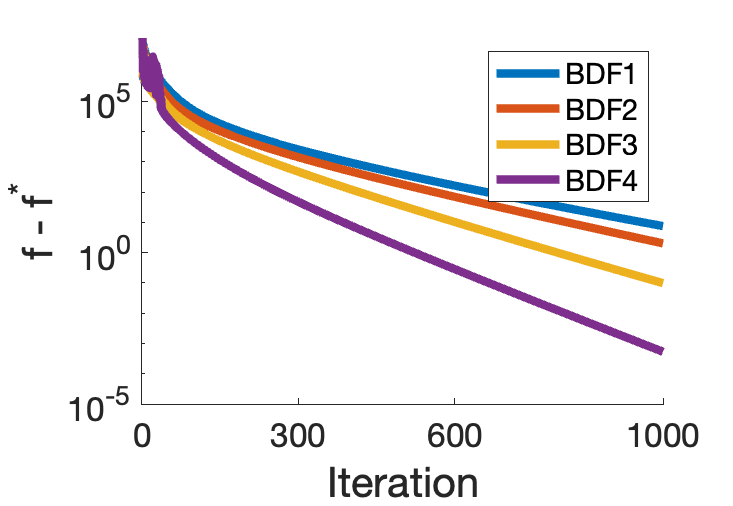}
\end{subfigure}
    \begin{subfigure}[b]{0.35\textwidth}    
    \includegraphics[width=\linewidth, trim={2ex .5ex 2ex 1ex}, clip]{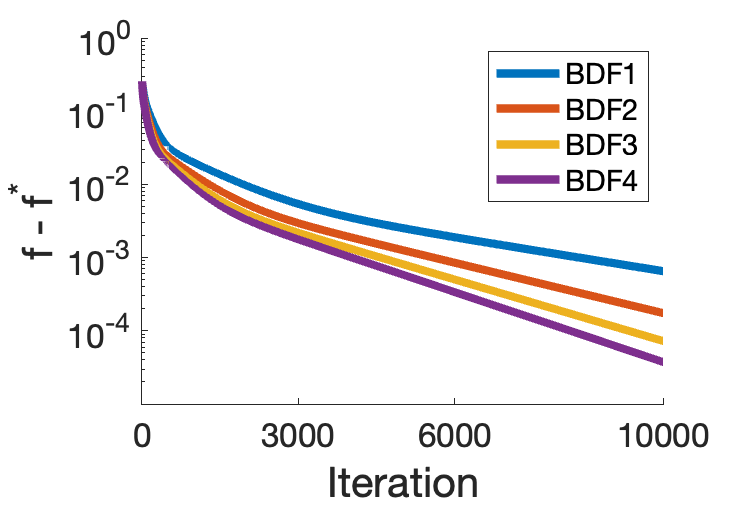}
\end{subfigure}
    \caption{Comparsion of different BDF schemes for   proximal gradient with $\ell_1$ penalty \textbf{(top left)}, 
     proximal gradient with LSP (nonconvex) penalty \textbf{(top right)}, alternating minimizations for matrix factorization  \textbf{(bottom left)},
     and alternating linear projections \textbf{(bottom right)}.
     In all cases, we use $m = 1$ inner iterations.
     }
     \label{fig:numerical}
\end{figure*}

From the above four experiments, we observe a surprising phenomenon that that the higher order methods   perform orders of magnitude better than the lower order methods, especially on problems that seem to be ill-conditioned.
\section{Analysis over smooth problems}
This section provides convergence results when $h(\vx) =0$, and we approximate the full multistep proximal point method using $m$ iterations of gradient descent. 
\subsection{Quadratic Analysis}
In this subsection, we show our convergence result by restricting $f(\vx) = \vx^T Q \vx $ and $h(\vx) = 0$ where $Q$ is a positive definite matrix with eigenvalue $\mu \leq \lambda_i \leq L$.
We define the radius of convergence of an iteration scheme $\vx^{(k)}$ as $\rho$ where $0\leq \rho\leq 1$ and
\begin{equation}
\|\vx^{(k)}-\vx^*\|_2\leq \rho^k\|\vx^{(0)}-\vx^*\|_2.
\label{eq:conv-linear-radius}
\end{equation}
We now consider an approximate multistep implicit method, 
with $m = 1$,
\begin{align*}
\vx^{(k+\tau + 1)} &= A^{m} \vx^{(k+\tau)}  + \frac{\beta}{\alpha}\sum_{j=1}^m A^{j-1} \sum_{i=1}^\tau \xi_i \vx^{(k+i)}\\
&\approx \argmin{\vx} \; \frac{1}{2} \vx^TQ\vx + \frac{1}{2\alpha } \|\vx - \sum_{i=1}^\tau \xi_i \vx^{k+i}\|_2^2
\end{align*}

where 
\begin{equation}
A =\left(1-\frac{\beta}{\alpha}\right) I-\beta Q.
\label{eq:helperA}
\end{equation}
Defining also 
\[
B:=\frac{\beta}{\alpha} \sum_{j=1}^m A^{j-1}
\]
allows a first-order view of the system
\[
\underbrace{\begin{bmatrix}
\vx^{(k+\tau+1)}\\
\vx^{(k+\tau)}\\
\vx^{(k+\tau-1)}\\
\vdots\\
\vx^{(k+2)}\\
\end{bmatrix}}_{\vz^{(k+1)}}
 \hspace{-1ex} = \hspace{-1ex}
 \setlength{\arraycolsep}{1mm}
\underbrace{\begin{bmatrix}
A^m +   \xi_{\tau}B & \xi_{\tau-1} B& \cdots &   \xi_1B\\
I&0&\cdots & 0  \\
0&I&\cdots & 0 \\
\vdots & \vdots & \ddots & \vdots  \\
0&0&\cdots  & 0 \\
\end{bmatrix}}_{=: M}
\hspace{-1ex}
\underbrace{\begin{bmatrix}
\vx^{(k+\tau)}\\
\vx^{(k+\tau-1)}\\
\vx^{(k+\tau-2)}\\
\vdots\\
\vx^{(k+1)}\\
\end{bmatrix}}_{\vz^{(k)}}
\]
and defining $\vz^* = \textbf{vec} (\vx^*,\vx^*,...,\vx^*)$
we have our usual contraction 
\[
(\vz^{(k)}-\vz^*) = M^k (\vz^{(0)}-\vz^*).
\]
Unfortunately, the 2-norm $\|M\|_2$ is bounded below by 1, making it difficult to prove contraction. However, invoking Gelfand's rule \citep{gelfand1941normierte}, we may show that
\begin{equation}
\|\vz^{(k)}-\vz^*\|_2 \lesssim (\lambda_{\max}(M))^k\|\vz^{(0)}-\vz^*\|_2.
\label{eq:contract-bdf}
\end{equation}
That is, although $M$ is not symmetric, \eqref{eq:contract-bdf} holds asymptotically.

In Tables \ref{tab:quad_sensitivity} and \ref{tab:quad_optimal}, we give the limits on step size $\beta$ and on optimal $\rho$ by finding values in which $\lambda_{\max}(M) < 1$, guaranteeing stability. Alongside, figure \ref{fig:multistep_poly_stability} shows $\rho$ over a continuation of $\beta$, over changing $m$ and $\alpha$. There is a clear correlation with greater stability and faster convergence with  increased $\alpha$ (implicit step size), as expected. The reliance on $m$ is interestingly unexpected; more $m$ leads to better solving of the inner prox step, but gives little effect on stability and can even give a negative effect on convergence. This suggests that though implicit methods are thought of as too expensive, very approximate versions are not only practical, they are very close to optimal. 

\begin{table}[ht]
    \centering
    \begin{tabular}{l|ccc}
&$L = 2$&$L = 10$&$L = 100$\\\hline
     \PPM (1)&0.667 & 0.182 & 0.0198 \\
     \PPM (10)&0.952 & 0.198 & 0.0200\\
     BDF2 (1)&0.665 & 0.181 & 0.0200   \\
     BDF2 (10)& 0.940 & 0.197 & 0.0200 \\
     BDF3 (1)&0.608 & 0.178 & 0.0200 \\
     BDF3 (10)&0.940 & 0.197 & 0.0200\\
\end{tabular}
    \caption{\textbf{Sensitivity of parameters.} Upper bound on best step size to avoid divergence $\beta$ for PPM and BDFs).  \PPM$(\alpha)$ means proximal point, with step size $\alpha$. (Bigger is better.) Here, we pick $m = 4$, and notice very little deviation for $m > 4$. Here, $\mu = 1$.}
    \label{tab:quad_sensitivity}
\end{table}

\begin{table}[ht]
    \centering
    \begin{tabular}{l|ccc}
&$L = 2$&$L = 10$&$L = 100$\\\hline
     \PPM (4,1)& 0.500 & 0.596 & 0.926\\
     \PPM (20,1)& 0.500 & 0.500 & 0.724\\
     \PPM (4, 10)& 0.0935 & 0.466 & 0.923\\
     \PPM (20,10) & 0.0909 & 0.100 & 0.676\\
     BDF2 (4,1)&0.326 & 0.282 & 0.905   \\
     BDF2 (20,1)&0.303 & 0.211 & 0.457   \\
     BDF2 (4,10)&0.059 & 0.423 & 0.941 \\
     BDF2 (20,10)&0.024 & 0.024 & 0.737 \\
     BDF3 (4,1)& 0.377 & 0.451 & 0.923\\
     BDF3 (20,1)&0.377 & 0.306 & 0.470 \\
     BDF3 (4,10)&0.197 & 0.459 & 0.943  \\
     BDF3 (20,10)&0.197 & 0.165 & 0.739  \\
\end{tabular}
    \caption{\textbf{Optimal $\rho$.} Minimum $\rho$ given best best step size, $\beta$ for PPM and BDFs). PPM$(m,\alpha)$ means proximal point, with step size $\alpha$ over $m$ iterations. (Smaller is better.) Here, $\mu = 1$.}
    \label{tab:quad_optimal}
\end{table}
\begin{figure*}
    \centering
    \includegraphics[width=.7\linewidth]{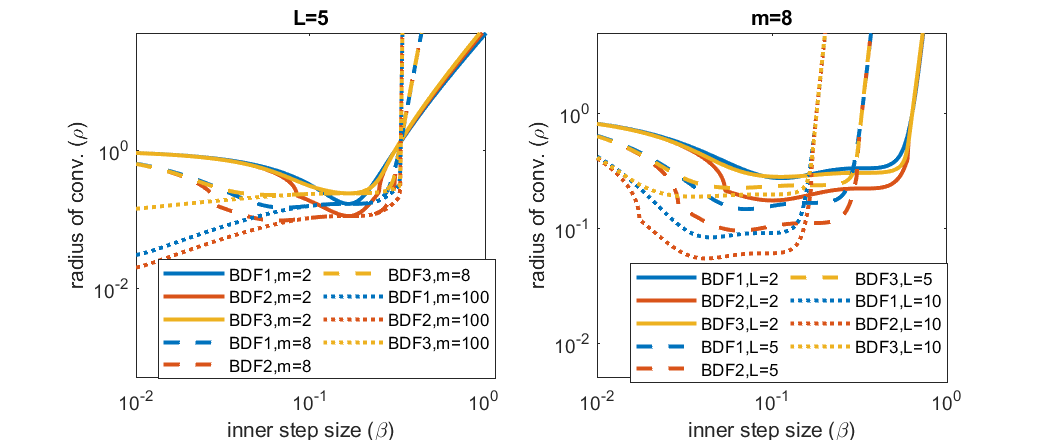}
    \caption{\textbf{Multistep radius of convergence.} $\lambda_{\max}(M)$ over changing values of $\beta$, for $\alpha = 1$ and varying $L = \lambda_{\max}(Q)$ and $m$ = number of inner gradient steps.}
    \label{fig:multistep_poly_stability}
\end{figure*}
\subsection{Convergence Result for Smooth Case}
This subsection provides convergence results for general $f(\vx)$. The first convergence result is when $\mathcal F(\tilde{\vx}_k) = \prox_{\alpha f}(\tilde{\vx}_k)$
\begin{theorem}
\label{th:exactbdm}
Assume $f$ is $\mu$-convex and differentiable everywhere.
The method by equation \ref{eq:main} converges at the following rates:

\begin{itemize}
    \item When $\mu > 0$ ($f$ is strongly convex),
    \[
\|\vx^{(k+\tau+1)}-\vx^*\|\leq C_0 \left (\frac{1}{1+\alpha \mu} \right )^{[\frac{k+\tau+1}{\tau}]}  
\]
where $C_0 = \max_{j \in \{ 0,1,\cdots \tau-1 \}} \Vert \vx^{(j)} - \vx^* \Vert.$

\item For general $\mu \leq 0$ and $\delta < 
\frac{1}{\tau-1}$,
\begin{equation}
    \min_{0\leq s \leq k } \Vert \nabla f(\vx^{(k+\tau+1)}) \Vert^2  \leq \frac{C_1}{\alpha k} + \frac{C_2 \delta}{\alpha^2k}
    \label{eq:bdmrate-general}
    \end{equation}
provided 
\[
0 \leq \alpha < \frac{2-2(\tau-1)\delta}{|\mu|}.
\]
    Here,
    \[
    C_1 = f(\vx^{(\tau)}) - f(\vx^*), \; C_2 = 2\sum_{s=0}^{\tau-1} \Vert \vx^{(s+1)}-\vx^{(s)} \Vert^2.
    \]

\item When $f$ is convex ($\mu = 0$), no upper bound restriction on $\alpha$ is needed.
\end{itemize}
\end{theorem}
\begin{proof}
The proof is in Appendix \ref{sec:app:convproofs}.    
\end{proof}
More specifically, when $\tau = 1$, the above result reduces the convergence result for the Proximal Point Method (PPM).
%
The exact multi-step proximal point methods, however, are not practical in practice, becasue the at each iteration you will need to solve an inner optimization problem which could be equally computationally expansive. In addition, even analytical solutions may have rounding error. As a result, it is neccessary to have a convergence result for approximate multi-step proximal point methods.
We now provide convergence results when the proximal point step is inexact. In this setting, we require an additional assumption.
\begin{assumption}[$\gamma$-contractive]
We say a method in  equation \eqref{eq:main} is \emph{$\gamma$-contractive} if, given $\vx$, it produces $\tilde \vx = \mathcal F(\vx)$ where
\[
\|\prox_{\alpha f}(\vx)-{\tilde \vx}\|\leq \gamma \|\prox_{\alpha f}(\vx)- \vx\|.\tag{$\gamma$C}
\]
\end{assumption}
The approximate results are given below:
\begin{theorem}
\label{th:approxbdm}
Consider $f$ everywhere differentiable and $\mu$-convex.
Suppose that a method in \ref{eq:main}  is $\gamma$-contractive. Then the Approximate Multi-Step Proximal Point converges at the following rates.

\begin{itemize}
    \item When $\mu > 0$ ($f$ is strongly convex) then for 
    \[
    \alpha < \frac{2\gamma }{\mu(1-\gamma)},
    \]
then
\[    \Vert \vx^{k+\tau+1} - \vx^* \Vert   \leq C_0\left ( \gamma +\frac{1+\gamma}{1+\alpha \mu} \right)^{[\frac{k+\tau+1}{\tau}]} 
    \]
    where $C_0 =  \max_{j \in \{ 0,1,\cdots \tau-1 \}} \Vert \vx^{j} - \vx^* \Vert$.

\item 
For general $\mu \leq 0$, when  $\delta < 
\frac{1}{2\tau-2}$, the same convergence rate 
\eqref{eq:bdmrate-general} holds, whenever
\[
\alpha < \min \left({\frac{2-2(\tau-1)\delta- 16\gamma^2}{|\mu|},\frac{1}{|\mu|}}\right).
\]
\item When $\mu = 0$, there is no upper bound on $\alpha$.
\end{itemize}

\end{theorem}
\begin{proof}The proof is in Appendix \ref{sec:app:convproofs}.
\end{proof}
Similarly, when $\tau = 1$, the results reduce to the  result for approximate PPM. From here we can see a clear tradeoff between $\gamma$, the requirement of contraction (smaller is better) and $\alpha$, the step size for faster convergence (bigger is better). The flexibility offered by this tradeoff is an advantage of these implicit methods; in a sense, in cases where it may be tolerable to allow a computation to run for longer at the benefit of less chance of catastrophic divergence, this new family of methods may prove advantageous.

\section{Conclusion and Future Work}
The goal of this work is to investigate the use of approximate implicit discretizations of \eqref{eq:prox_flow} in badly conditioned non-smooth problem settings. In this work, we found that the higher-order approximate implicit discretization helps in many optimization problems. However, it is worth pointing out that it is also important to find an efficient manner of choosing the approximate methods. In our work, we provide multiple approaches that approximate the implicit updates, and both work well in practice. 

\newpage
\bibliography{refs.bib}

\newpage
\appendix

\onecolumn
 \section{CONVERGENCE PROOFS}
\label{sec:app:convproofs}

\begin{assumption}
$f(\vx)$ is convex and continuously differentiable:
\begin{align*}
    f(\vy) \geq f(\vx) + \langle \nabla f(\vx),\vy - \vx\rangle
\end{align*}
\end{assumption}
\begin{assumption}
$f(\vx)$ is  $\mu$- strongly convex:
\begin{align*}
    f(\vy) \geq f(\vx) + \langle\nabla f(\vx), \vy - \vx\rangle + \frac{\mu}{2}\Vert \vy - \vx \Vert^2
\end{align*}
This assumption is needed when for linear convergence.
\end{assumption}
\begin{assumption}
$f(\vx)$ is a $L$- smooth function:
\begin{equation*}
    \Vert \nabla f(\vy) - \nabla f(\vx) \Vert \leq L \Vert \vy - \vx \Vert 
\end{equation*}
\end{assumption}

In this section, we are giving the main result we have for the multistep algorithm. We first need a technical lemma:
\begin{lemma}\label{le1}
Let $\{ a_n \}_{n=0}^{\infty}$ be a $\tau$ step sublinear sequence:
\begin{equation}
    a_n \leq \sum_{i=1}^{\tau} \xi_i a_{n-i}
    \label{eq:assp:sublinear}
\end{equation}
where $\xi_i \geq 0$ and $n \geq \tau$. Then there exists $n >i_1 > i_2 > \cdots  i_{[\frac{n}{\tau}]}$ such that 
\begin{equation}
    a_n \leq  \left ( \sum_{i=1}^{n}\xi_i \right)  a_{i_1} \leq \left ( \sum_{i=1}^{n}\xi_i \right)^2  a_{i_2} \leq 
     \cdots\leq  \left ( \sum_{i=1}^{n}\xi_i \right)^{[\frac{n}{\tau}]}  a_{i_{[\frac{n}{\tau}}]}
     \label{eq:lem:sublinear}
\end{equation}
for all $k \in \{1,2,\cdots,[\frac{n}{\tau}] \}$
\end{lemma}
\begin{proof}
We prove this result by mathematical induction. When $n = \tau$, by using Holder's inequality for $p=1,q=\infty$, we will have:
\begin{align*}
    a_{\tau} \leq  \sum_{i=1}^{\tau} \xi_i a_{n-i} \leq \left (\sum_{i=1}^{\tau}\xi_i \right) \left (\max_{i\in \{1,2,\cdots,\tau \}} a_{\tau-i} \right ).
\end{align*} In this first step, we pick $i_1 = \argmax{i\in \{1,2,\cdots,\tau \}} a_{\tau-i}$. This gives
\[
a_n \leq  \left ( \sum_{i=1}^{\tau}\xi_i \right)  a_{i_1}, \qquad i_1 = \argmax{i\in \{1,2,\cdots,\tau \}} a_{\tau-i}.
\]
Now applying \eqref{eq:assp:sublinear} to $a_{i_1}$, we can further conclude
\[
a_{i_1} \leq  \left ( \sum_{i=1}^{\tau}\xi_i \right)  a_{i_2},\qquad i_2 = i_1-j, \quad j =  \argmax{i\in 1,2,...,\tau} a_{i_1-i}.
\]
This we can apply recursively, to get
\[
a_{i_{k}} \leq  \left ( \sum_{i=1}^{\tau}\xi_i \right)  a_{i_{k+1}},\qquad i_{k+1} = i_k-j, \quad j =  \argmax{i\in 1,2,...,\tau} a_{i_k-i}.
\]


This inequality can be chained for as long as there exists $i_k \geq 0$. Note that at each step, the smallest value $i_k$ can take is $n - \tau k$. So, the largest value of $k$ such that $i_k \geq 0$ is guaranteed is $k = \lceil\frac{n}{\tau}\rceil$.
\end{proof}

\begin{remark}
Lemma \ref{le1} holds for some $i_{[\frac{n}{\tau}]} \in \{0,1,2,\cdots,\tau-1\}$.
\end{remark}
Now we give the theorem which shows the linear convergence of the algorithm in $\mu$ strongly convex case:
\begin{theorem} \label{mu_strong_case}
Let $f(\vx)$ be a $\mu$-strongly convex and continuously differentiable function.  Suppose that we minimize $f(\vx)$ using a $\tau$ multi-step proximal point methods,
where we assume $\sum_{i=1}^n \xi_i = 1$ and $\xi_i \geq 0$. Then it will converge linearly
\[
\|\vx^{k+\tau+1}-\vx^*\|\leq \left (\frac{1}{1+\alpha \mu} \right )^{[\frac{k+\tau+1}{\tau}]}  \max_{j \in \{ 0,1,\cdots \tau-1 \}} \Vert \vx^{j} - \vx^* \Vert.
\]
\label{th:stronglyconvex}
\end{theorem}
\begin{proof}
Recall that 

\[
  \vx_{k+\tau + 1} =  \text{prox}_{\alpha f }(\tilde{\vx}_{k+\tau}).
\]
which is equivalent to 
\begin{equation}\label{bdf}
    \vx_{k+\tau+1} = \sum_{i=1}^{\tau} \vx^{k+i} - \alpha \nabla f(\vx^{k+\tau+1})
\end{equation}
We first observe that 
\begin{align*}
 \Vert \sum_{i=1}^{\tau}\xi_i (\vx^{k+i}-\vx^*)\Vert ^2 &\overset{\eqref{bdf}}{=} \Vert \vx^{k+\tau+1} - \vx^* + \alpha \nabla f(\vx^{k+\tau+1})\Vert^2 \\
 &=\Vert \vx^{k+\tau+1} - \vx^* \Vert^2 + \alpha^2 \Vert \nabla f(\vx^{k+\tau+1}) \Vert^2  + 2\alpha \langle \nabla f(\vx^{k+\tau+1}), \vx^{k+\tau+1} - \vx^*\rangle \\
 &\overset{(a)}{ \geq} \Vert \vx^{k+\tau+1} - \vx^* \Vert^2 +\alpha^2 \mu^2 \Vert \vx^{k+\tau+1} - \vx^*\Vert^2 +2\alpha \langle \nabla f(\vx^{k+\tau+1}), \vx^{k+\tau+1} - \vx^*\rangle \\
 & \overset{(b)}{\geq} (1+\alpha^2 \mu^2) \Vert \vx^{k+\tau+1} - \vx^* \Vert^2 + 2\alpha(\mu\Vert \vx^{k+\tau+1} - \vx^* \Vert^2) \\
 & = (1+2\alpha \mu+\alpha^2 \mu^2) \Vert \vx^{k+\tau+1} - \vx^* \Vert^2
\end{align*}
where (a) and (b) are consequences of $\mu$-strong convexity.
As a result, we will have:
\begin{align}
\Vert \vx^{k+\tau+1} - \vx^* \Vert &\leq \frac{1}{1+\alpha \mu} \Vert \sum_{i=1}^{\tau}\xi_i (\vx^{k+i}-\vx^*)\Vert   \label{eq:strongconvex_1stepcontract}\\
&\leq \frac{1}{1+\alpha \mu}  \sum_{i=1}^{\tau}\xi_i \Vert \vx^{k+i}-\vx^*\Vert 
\end{align}
by convexity of the $\ell_2$ norm.
Now by \textbf{Lemma} \ref{le1}, we will have $i_1,i_2,\cdots,i_{[\frac{k+\tau+1}{\tau}]}$ such that 
\begin{align*}
    \Vert \vx^{k+\tau+1} - \vx^* \Vert &\leq  \left (\frac{1}{1+\alpha \mu} \right )  \Vert \vx^{i_1} - \vx^* \Vert \leq \left (\frac{1}{1+\alpha \mu} \right )^2  \Vert \vx^{i_2} - \vx^* \Vert \\
    & \cdots \left (\frac{1}{1+\alpha \mu} \right )^{[\frac{k+\tau+1}{\tau}]}  \Vert \vx^{i_{[\frac{k+\tau+1}{\tau}]}} - \vx^* \Vert \\
    & \leq \left (\frac{1}{1+\alpha \mu} \right )^{[\frac{k+\tau+1}{\tau}]}  \max_{j \in \{ 0,1,\cdots \tau-1 \}} \Vert \vx^{j} - \vx^* \Vert \label{eq11}
\end{align*}
for all $k \in \{1,2,\cdots,[\frac{n}{\tau}] \}$ where $m \in \{i_1,i_2,\cdots,i_{[\frac{k+\tau+1}{\tau}]} \}$
Hence, we prove the desired result.
\end{proof}
According to Theorem \ref{th:stronglyconvex}
we can have the following corollary:
\begin{corollary}
Let $f(\vx)$ be a $\mu$-strongly convex and continuously differentiable function. Then, the proximal point method
\begin{align*}
    \vx^{k+1} = \argmin{\vx} \left (f(\vx) + \frac{1}{2\alpha} \Vert \vx - \vx^{k} \Vert^2 \right )
\end{align*}
converges linearly 

\[
\|\vx^{k}-\vx^*\|\leq \left (\frac{1}{1+\alpha \mu} \right )^{k}    \Vert \vx^{0} - \vx^* \Vert
\]
with any step size $\alpha > 0$.
\end{corollary}
\begin{remark}
Although the above theorem only shows the case that $\xi_i\geq 0$
we can still have
linear convergence 
when $\xi_i<0$
by letting $\alpha \geq \frac{\sum_{i=1}^{\tau}\xi_i - 1}{\mu}$. See the following result.
\end{remark}
\begin{theorem}
Let $f(\vx)$ be a $\mu$-strongly convex and continuously differentiable function.  Suppose that we minimize $f(\vx)$ using a $\tau$ multi-step proximal point methods, \ref{bdf} 
where we assume $\sum_{i=1}^n \xi_i = 1$. Then the $\tau$ multi-step proximal point methods will converge linearly
\[
\|\vx^{k+\tau+1}-\vx^*\|\leq \left (\frac{\sum_{i=1}^{\tau} \vert \xi_i \vert }{1+\alpha \mu} \right )^{[\frac{k+\tau+1}{\tau}]}  \max_{j \in \{ 0,1,\cdots \tau-1 \}} \Vert \vx^{j} - \vx^* \Vert.
\]
\label{th:stronglyconvex-2}
\end{theorem}
\begin{proof}
Similar to \ref{mu_strong_case}, we will have:
\begin{align}
\Vert \vx^{k+\tau+1} - \vx^* \Vert &\leq \frac{1}{1+\alpha \mu} \Vert \sum_{i=1}^{\tau}\xi_i (\vx^{k+i}-\vx^*)\Vert   \\
&\leq \frac{\sum_{i=1}^{\tau} \vert \xi_i \vert}{1+\alpha \mu}  \sum_{i=1}^{\tau} \frac{\vert \xi_i  \vert}{\sum_{i=1}^{\tau} \vert \xi_i \vert} \Vert \vx^{k+i}-\vx^*\Vert 
\end{align}
 Now by \textbf{Lemma} \ref{le1} and similar to \ref{mu_strong_case}, we will have the desired result.
 
\end{proof}
\begin{remark}
Although the above result gives a linear convergence rate when the step size is large enough, the exact multi-step proximal point methods are in general  hard to solve to completion, and is usually approximated in practice. Below, we give the linear convergence rate for approximated multi-step proximal point methods.
\end{remark}
\begin{theorem}
\label{th:approx-strongly-convex}
Let $f(\vx)$ be a $\mu$-strongly convex and continuously differentiable function.  Suppose that the exact solution to
$\tau$ multi-step proximal point methods are 
\begin{align}
\tilde{\vx}^{k+\tau+1} = \sum_{i=1}^{\tau}\xi_i \vx^{k+i} - \alpha \nabla f(\tilde{\vx}^{k+\tau+1})
\end{align}
where  $\sum_{i=1}^n \xi_i = 1$ and $\xi \geq 0 $. Suppose that an approximate solution $\vx^{k+\tau +1}$ of $\tilde{\vx}^{k+\tau+1}$ satisfies a contraction :
\begin{equation}
    \Vert \vx^{k+\tau +1} - \tilde{\vx}^{k+\tau+1} \Vert  \leq \gamma     \Vert \sum_{i=1}^{\tau}\xi_i \vx^{k+i} - \tilde{\vx}^{k+\tau+1} \Vert \label{eq:approx_contract_assp}
\end{equation}
where $\gamma < 1$.Then $\vx^{k+\tau+1}$ converges to $\vx^*$ linaerly

\[    \Vert \vx^{k+\tau+1} - \vx^* \Vert   \leq \left ( \gamma +\frac{1+\gamma}{1+\alpha \mu} \right)^{[\frac{k+\tau+1}{\tau}]}  \max_{j \in \{ 0,1,\cdots \tau-1 \}} \Vert \vx^{j} - \vx^* \Vert
    \]
when $\gamma < \frac{\alpha \mu }{2+\alpha \mu}$.

\end{theorem}
\begin{proof}

 order seems a bit illogical, how about
\begin{eqnarray*}
    \Vert \vx^{k+\tau +1} - \vx^* \Vert & \leq&  \Vert \vx^{k+\tau +1} - \tilde{\vx}^{k+\tau+1} \Vert  + \Vert \tilde{\vx}^{k+\tau+1} -\vx^{*} \Vert \\
    &\overset{\eqref{eq:approx_contract_assp}}{\leq} & \gamma     \Vert \sum_{i=1}^{\tau}\xi_i \vx^{k+i} - \tilde{\vx}^{k+\tau+1}\Vert + \Vert \tilde{\vx}^{k+\tau+1} -\vx^{*} \Vert\\
    & \le& \gamma \left( \Vert  \sum_{i=1}^{\tau}\xi_i \vx^{k+i} - \vx^*\Vert +\Vert \tilde{\vx}^{k+\tau+1}  - \vx^*\Vert \right ) +\Vert \tilde{\vx}^{k+\tau+1} -\vx^{*} \Vert\\
    &\overset{(a)}{\leq}& (\gamma +\frac{1+\gamma}{1+\alpha \mu})\sum_{i=1}^{\tau}\xi_i \Vert \vx^{k+i}-\vx^*\Vert  \\
\end{eqnarray*}
where (a) follows from \eqref{eq:strongconvex_1stepcontract}

Now by \textbf{Lemma}\ref{le1}, we have
\begin{align*}
    \Vert \vx^{k+\tau+1} - \vx^* \Vert 
     \leq \left ( \gamma +\frac{1+\gamma}{1+\alpha \mu} \right)^{[\frac{k+\tau+1}{\tau}]}  \max_{j \in \{ 0,1,\cdots \tau-1 \}} \Vert \vx^{j} - \vx^* \Vert
\end{align*}for all $k \in \{1,2,\cdots,[\frac{n}{\tau}] \}$.
Hence, we prove the desired result.
\end{proof}
\begin{corollary}
Let $f(\vx)$ be a $\mu$-strongly convex and $L$-smooth function.
Suppose that we use a $\tau$multi-step proximal point methods with update (\ref{bdf}),where \eqref{bdf} is computed approximately
by using $n$ steps of gradient descent with step size $\beta  \leq \frac{\alpha}{\alpha L+1}$. 
This method
will converge linearly when $\alpha$ is large enough with rate

\[    \Vert \vx^{k+\tau+1} - \vx^* \Vert   \leq \left ( (1-\beta \mu -\frac{\beta}{\alpha})^n +\frac{1+(1-\beta \mu -\frac{\beta}{\alpha})^n}{1+\alpha \mu} \right)^{[\frac{k+\tau+1}{\tau}]}  \max_{j \in \{ 0,1,\cdots \tau-1 \}} \Vert \vx^{j} - \vx^* \Vert
    \]

\end{corollary}
\begin{proof}
Recall that for a $\rho$-strongly convex function $g$, then $n$ steps of gradient descent with appropriate step size $\beta$ gives convergence guarantee
\[
\|\vx^n - \vx^*\| \leq (1-\rho\beta)^n\|\vx^0-\vx^*\|.
\]
Now, note that 
\[
g(\vx) =  \frac{1}{2\alpha}\Vert \vx -\sum_{i=1}^{\tau}\xi_i \vx^{k+i}\Vert^2+ f(\vx)
\]
 is $\rho = \mu + \frac{1}{\alpha}$ strongly convex function. Then, using $n$ steps of gradient descent with starting 
 point $\sum_{i=1}^{\tau}\xi_i \vx^{k+i}$ 
 gives
\begin{align*}
    \Vert \vx^{k+\tau +1}_{n} - \tilde{\vx}^{k+\tau+1}\Vert  \leq (1-\beta \mu -\frac{\beta}{\alpha})^n \Vert \sum_{i=1}^{\tau}\xi_i \vx^{k+i} - \tilde{\vx}^{k+\tau+1}\Vert 
\end{align*}
where $\tilde{\vx}^{k+\tau+1}$ fully minimizes $g$.
As a result, \eqref{eq:approx_contract_assp} is satisfied with $\gamma = (1-\beta \mu -\frac{\beta}{\alpha})^n$, and by invoking Theorem \ref{th:approx-strongly-convex}, we will have the desired result. 
\end{proof}
The following two theorems are for the case when $f(\vx)$ is $L-$ weakly convex for both exact and inexact case. We will also need the following simple fact:
\begin{fact}
\[
\|\sum_{i=1}^n a_i\|^2 \leq n\sum_{i=1}^n \|a_i\|^2
\]
\label{fact:variance}
\end{fact}
\begin{theorem}\label{Non_convex}
Suppose that $f(\vx)$ is an $L-$ weakly convex function. Then for any fixed step size $\alpha >0$, the update rule from (\ref{bdf}),
given $\xi_i$, will have:
\begin{align*}
    \min_{0\leq s \leq k } \Vert \nabla f(\vx^{k+\tau+1}) \Vert  = O(\frac{1}{\sqrt{k}})
\end{align*}
if $\alpha < \frac{2-2(\tau-1)\delta}{L}$ and $\delta < 
\frac{1}{\tau-1}$ where 
\[
\delta = (\tau-1) \sum_{j=1}^{\tau-1} \left  (\sum_{i=1}^{j}(\tau-i )\xi_i^2 \right).
\]
\end{theorem}
\begin{proof}
By $L$-smoothness, we will have:
\begin{align*}
    f(\vx^{k+\tau}) &\geq f(\vx^{k+\tau+1}) + \langle\nabla f(\vx^{k+\tau+1}),\vx^{k+\tau}-\vx^{k+\tau+1}\rangle - \frac{L}{2}      
     \Vert \vx^{k+\tau+1}-\vx^{k+\tau} \Vert^2 \\
     &= f(\vx^{k+\tau+1}) + \frac{1}{\alpha}\langle \sum_{i=1}^{\tau}\xi_i \vx^{k+i} -   \vx^{k+\tau+1},\vx^{k+\tau}-\vx^{k+\tau+1}\rangle - \frac{L}{2}      
     \Vert \vx^{k+\tau+1}-\vx^{k+\tau} \Vert^2 \\
     &\overset{(*)}{ =} 
    f(\vx^{k+\tau+1}) + \frac{1}{\alpha}\underbrace{\Vert \vx^{k+\tau+1} - \sum_{i=1}^{\tau}\xi_i \vx^{k+i}\Vert^2}_{= \alpha \Vert \nabla f(\vx^{(k+\tau+1)}) \Vert^2} + \left(\frac{1}{\alpha}-\frac{L}{2}\right)\Vert \vx^{k+\tau+1} - \vx^{k+\tau} \Vert^2 - \frac{1}{\alpha}\Vert \sum_{i=1}^{\tau-1} \xi_i(\vx^{k+i} -\vx^{k+\tau}) \Vert^2  ) 
\end{align*}
where (*) is $2a^Tb =\|a\|^2+\|b\|^2-\|a-b\|^2$.
Telescoping, 
\[
f(\vx^\tau) - f(\vx^*) \geq \alpha\sum_{k=0}^T \Vert\nabla f(\vx^{(k+\tau+1)})\Vert^2 + \left(\frac{1}{\alpha}-\frac{L}{2}\right)\sum_{k=0}^T \Vert \vx^{k+\tau+1} - \vx^{k+\tau} \Vert^2 - \frac{1}{\alpha}\sum_{k=0}^T \Vert \sum_{i=1}^{\tau-1} \xi_i(\vx^{k+i} -\vx^{k+\tau}) \Vert^2  
\]

Now from Fact \ref{fact:variance},
notice that:
\begin{align*}
    \Vert \sum_{i=1}^{\tau-1} \xi_i(\vx^{k+i} -\vx^{k+\tau}) \Vert^2 &\leq(\tau-1) \sum_{i=1}^{\tau-1}\xi_i^2\Vert \vx^{k+i}-\vx^{k+\tau} \Vert^2 \\
    &\overset{(*)}{\leq} (\tau-1) \sum_{i=1}^{\tau-1}\xi_i^2\|\sum_{j=i}^{\tau-1}\vx^{(k+j)} - \vx^{(k+j-1)}\|^2 \\
    &\overset{\text{Fact \ref{fact:variance}}}{\leq} 
     (\tau-1) \sum_{i=1}^{\tau-1}(\tau-i )\sum_{j=i}^{\tau-1}\xi_i^2\Vert \vx^{k+j}-\vx^{k+j+1} \Vert^2 \\
    &= (\tau-1) \sum_{j=1}^{\tau-1}  \left  (\sum_{i=1}^{j} (\tau-i)\xi_i^2 \right)\Vert \vx^{k+j}-\vx^{k+j+1} \Vert^2 \\
\end{align*}
where (*) is from reverse telescoping
\[
\vx^{(k+i)} - \vx^{(k+\tau)} = \sum_{j=i}^{\tau-1}(\vx^{(k+j)} - \vx^{(k+j-1)})
\]
Hence for $c_j =  \sum_{i=1}^{j} (\tau-i)\xi_i^2 $ we have:
\begin{align*}
 \sum_{s=0}^{k} \Vert \sum_{i=1}^{\tau-1} \xi_i(\vx^{s+i} -\vx^{s+\tau}) \Vert^2 
 &\leq (\tau-1) \sum_{r=0}^{k}\sum_{j=1}^{\tau-1} c_j \Vert \vx^{r+j}-\vx^{r+j+1} \Vert^2  \\ 
 &= (\tau-1)\sum_{j=1}^{\tau-1} c_j  \sum_{r=0}^{k}\Vert \vx^{r+j}-\vx^{r+j+1} \Vert^2  \\ 
 &\overset{s=r+j}{=} (\tau-1)\sum_{j=1}^{\tau-1} c_j  \sum_{s=j}^{j+k}\Vert \vx^{s}-\vx^{s+1} \Vert^2  \\ 
 & \leq \underbrace{(\tau-1)    \sum_{j=1}^{\tau-1}  c_j }_{=\delta}\sum_{s=0}^{k+\tau-1} \Vert \vx^{s+1}-\vx^{s} \Vert^2\\ 
\end{align*}
As a result, we will have:
\[
f(\vx^\tau) - f(\vx^*) \geq \alpha\sum_{k=0}^T \Vert\nabla f(\vx^{(k+\tau+1)})\Vert^2 + \left(\frac{1}{\alpha}-\frac{L}{2}-\frac{\delta}{\alpha}\right)\sum_{k=0}^T \Vert \vx^{k+\tau+1} - \vx^{k+\tau} \Vert^2 - \frac{\delta}{\alpha}\sum_{s=0}^{\tau-1} \Vert \vx^{s+1}-\vx^{s} \Vert^2
\]
Now suppose we pick $\alpha,\delta$ such that $\frac{1}{\alpha}-\frac{L}{2}- \frac{\delta}{\alpha} \geq 0$ (e.g. $\alpha < \frac{2-2(\tau-1)\delta}{L}$) then
\[
\frac{1}{T}\sum_{k=0}^T \Vert\nabla f(\vx^{(k+\tau+1)})\Vert^2   \leq \frac{1}{T\alpha}(f(\vx^\tau) - f(\vx^*)) + \frac{\delta}{ T\alpha^2}\sum_{s=0}^{\tau-1} \Vert \vx^{s+1}-\vx^{s} \Vert^2  = O(1/T).
\]
and we have the desired result.
\end{proof}
\begin{theorem}
Suppose that $f(\vx)$ is an $L$-weakly convex function. Then for any fixed step size $\alpha >0$,suppose that the exact solution to
$\tau$ multi-step proximal points methods are 
\begin{align}
\tilde{\vx}^{k+\tau+1} = \sum_{i=1}^{\tau}\xi_i \vx^{k+i} - \alpha \nabla f(\tilde{\vx}^{k+\tau+1})
\end{align}
where  $\sum_{i=1}^n \xi_i = 1$ and $\xi \geq 0 $. Suppose that an approximate solution $\vx^{k+\tau +1}$ of $\tilde{\vx}^{k+\tau+1}$ satisfies the following contraction with respect to $\vx^{k+\tau}$:
\begin{equation}
    \Vert \vx^{k+\tau +1} - \tilde{\vx}^{k+\tau+1} \Vert  \leq \gamma     \Vert  \vx^{k+\tau+1} - \vx^{k+\tau} \Vert \label{eq:approx_contract_assp-2}
\end{equation}
Then, we will have:
\begin{align*}
    \min_{0\leq s \leq k } \Vert \nabla f(\vx^{k+\tau+1}) \Vert  = O(\frac{1}{\sqrt{k}})
\end{align*}
if $\alpha < \min ({\frac{2-2(\tau-1)\delta- 16\gamma^2}{L},\frac{1}{L}})$ and $\delta < 
\frac{1}{2\tau-2}$ where 
\[
\delta = (\tau-1) \sum_{j=1}^{\tau-1} \left  (\sum_{i=1}^{j}(\tau-1-i )\xi_i^2 \right).
\]
\end{theorem}
\begin{proof}
By $L$-smoothness, we will have:
\begin{align*}
    f(\vx^{k+\tau}) &\geq f(\vx^{k+\tau+1}) +  \frac{1}{\alpha}\langle\ \alpha \nabla f(\vx^{k+\tau+1}),\vx^{k+\tau}-\vx^{k+\tau+1}\rangle - \frac{L}{2}      
     \Vert \vx^{k+\tau+1}-\vx^{k+\tau} \Vert^2  \\
     &\overset{(*)}{ =} 
    f(\vx^{k+\tau+1}) + \alpha \Vert \nabla f(\vx^{(k+\tau+1)}) \Vert^2 + 
     \left(\frac{1}{\alpha}-\frac{L}{2}\right)\Vert \vx^{k+\tau+1}  - \vx^{k+\tau} \Vert^2  \\ & - \frac{1}{\alpha}\Vert \vx^{k+\tau} -\vx^{k+\tau+1} - \alpha \nabla f(\vx^{k+\tau+1}) \Vert^2  ) 
     \\
     &=  f(\vx^{k+\tau+1}) + \alpha \Vert \nabla f(\vx^{(k+\tau+1)}) \Vert^2 + 
     \left(\frac{1}{\alpha}-\frac{L}{2}\right)\Vert \vx^{k+\tau+1} - \vx^{k+\tau} \Vert^2 \\&  - \frac{1}{\alpha}\Vert \vx^{k+\tau} -\vx^{k+\tau+1} - \alpha \nabla f(\vx^{k+\tau+1}) + \tilde{\vx}^{k+\tau+1} - \sum_{i=1}^{\tau}\xi_i \vx^{k+i} + \alpha   \nabla f(\tilde{\vx}^{k+\tau+1}) \Vert^2  \\ 
     & \geq   f(\vx^{k+\tau+1}) + \alpha \Vert \nabla f(\vx^{(k+\tau+1)}) \Vert^2 + 
     \left(\frac{1}{\alpha}-\frac{L}{2}\right)\Vert \vx^{k+\tau+1} - \vx^{k+\tau} \Vert^2 \\&  - \frac{2}{\alpha}\Vert \sum_{i=1}^{\tau-1} \xi_i(\vx^{k+i} -\vx^{k+\tau}) \Vert^2 - \frac{2}{\alpha} \Vert \vx^{k+\tau+1} + \alpha \nabla f(\vx^{k+\tau+1}) - \tilde{\vx}^{k+\tau+1}  - \alpha   \nabla f(\tilde{\vx}^{k+\tau+1}) \Vert^2 \\
     & \geq  f(\vx^{k+\tau+1}) + \alpha \Vert \nabla f(\vx^{(k+\tau+1)}) \Vert^2 + 
     \left(\frac{1}{\alpha}-\frac{L}{2}\right)\Vert \vx^{k+\tau+1} - \vx^{k+\tau} \Vert^2 \\&  - \frac{2}{\alpha}\Vert \sum_{i=1}^{\tau-1} \xi_i(\vx^{k+i} -\vx^{k+\tau}) \Vert^2 - \frac{2(1+\alpha L)^2}{\alpha} \Vert \vx^{k+\tau+1}  - \tilde{\vx}^{k+\tau+1}  \Vert^2 \\
     &  \geq  f(\vx^{k+\tau+1}) + \alpha \Vert \nabla f(\vx^{(k+\tau+1)}) \Vert^2 + 
     \left(\frac{1}{\alpha}-\frac{L}{2}\right)\Vert \vx^{k+\tau+1} - \vx^{k+\tau} \Vert^2 \\&  - \frac{2}{\alpha}\Vert \sum_{i=1}^{\tau-1} \xi_i(\vx^{k+i} -\vx^{k+\tau}) \Vert^2 - \frac{2\gamma^2 (1+\alpha L)^2}{\alpha} \Vert \vx^{k+\tau+1}  - \vx^{k+\tau}  \Vert^2
\end{align*}
Now we repeat the step from \ref{Non_convex}, we will have:
\[
f(\vx^\tau) - f(\vx^*) \geq \alpha \sum_{k=0}^T \Vert\nabla f(\vx^{(k+\tau+1)})\Vert^2 + \left(\frac{1}{\alpha}-\frac{L}{2} - \frac{2\gamma^2 (1+\alpha L)^2}{\alpha}-\frac{2\delta}{\alpha}\right)\sum_{k=0}^T \Vert \vx^{k+\tau+1} - \vx^{k+\tau} \Vert^2 - \frac{2\delta}{\alpha}\sum_{s=0}^{\tau-1} \Vert \vx^{s+1}-\vx^{s} \Vert^2
\]
Hence, we prove the desired result.
\end{proof}
\section{ADDITIONAL EXPERIMENT}


  Figure \ref{fig:numerical5} shows the performance of various BDF schemes over our applications,f or $m = 5$ inner iterations. Note that the increased number of inner iterations do not have discernable effect; the results look very similar to that of Figure \ref{fig:numerical} ($m=1$). 
\begin{figure}[!ht]
    \centering
         \begin{subfigure}[b]{0.3\textwidth}
 \includegraphics[width=\linewidth, trim={2ex .5ex 2ex 1ex}, clip]{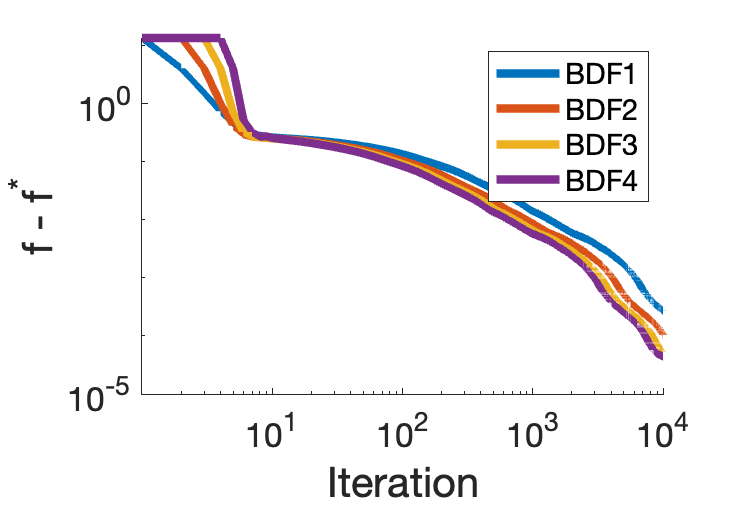}
\end{subfigure}
         \begin{subfigure}[b]{0.3\textwidth}
    \includegraphics[width=\linewidth, trim={2ex .5ex 2ex 1ex}, clip]{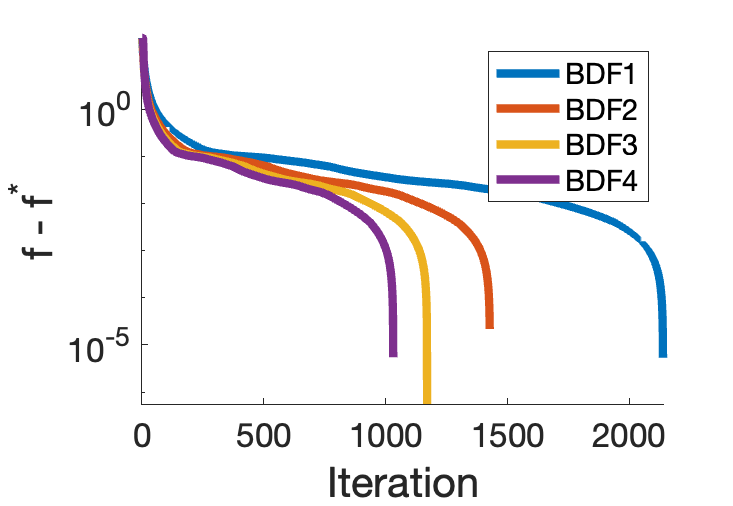}
\end{subfigure}
         \begin{subfigure}[b]{0.3\textwidth}    
    \includegraphics[width=\linewidth, trim={2ex .5ex 2ex 1ex}, clip]{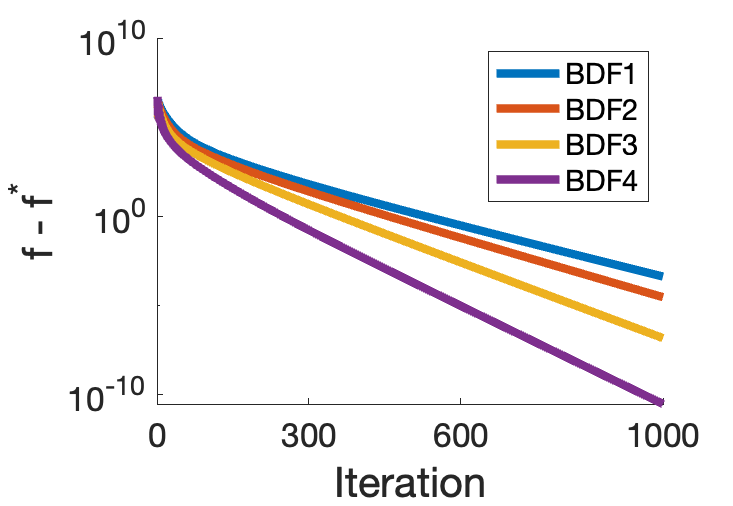}
\end{subfigure}
    \caption{Comparsion of different BDF schemes for   proximal gradient with $\ell_1$ penalty \textbf{(left)}, 
     proximal gradient with LSP (nonconvex) penalty \textbf{(middle)}, and alternating minimizations for matrix factorization  \textbf{(right)}
     In all cases, we use $m = 5$ inner iterations.
     }
     \label{fig:numerical5}
\end{figure}
\end{document}